\documentclass[11pt]{amsart}

\usepackage[T1]{fontenc}
\usepackage[latin1]{inputenc}
\usepackage[english]{babel}
\usepackage{amsmath,amssymb,amsthm}
\usepackage{enumitem}
\usepackage{hyperref}

\newtheorem{theorem}{Theorem}[section]
\newtheorem{lemma}[theorem]{Lemma}
\newtheorem{proposition}[theorem]{Proposition}
\newtheorem{corollary}[theorem]{Corollary}

\theoremstyle{definition}

\newtheorem{question}[theorem]{Question}
\newtheorem{remark}[theorem]{Remark}
\newtheorem*{remark*}{Remark}

\newcommand{\id}{\operatorname{id}}
\newcommand{\Aut}{\operatorname{Aut}}
\newcommand{\mcc}{M\textsuperscript{c}Carthy}

\renewcommand{\MR}[1]{}

\title[WCOs on the ball]{Weighted composition operators on Hilbert function spaces on the ball}
\author[M. Hartz]{Michael Hartz}
\address{Fachrichtung Mathematik, Universit\"at des Saarlandes, 66123 Saarbr\"ucken, Germany}
\email{hartz@math.uni-sb.de}
\thanks{M.H. was partially supported by the Emmy Noether Program
of the German Research Foundation (DFG Grant 466012782). }

\author[M. Tornes]{Maximilian Tornes}
\address{Department of Mathematics, University of Manitoba, Winnipeg, Canada}
\email{tornesm@myumanitoba.ca}
\date{\today}


\begin{document}

\begin{abstract}
  A weighted composition operator on a reproducing kernel Hilbert space is given by a composition, followed by a multiplication.
  We study unitary and co-isometric weighted composition operators on unitarily invariant spaces on the Euclidean unit ball $\mathbb B_d$.
  We establish a dichotomy between the spaces $\mathcal{H}_\gamma$ with reproducing kernel $(1 - \langle z,w \rangle)^{-\gamma}$ for $\gamma > 0$, and all other spaces. Whereas the former admit many unitary weighted composition operators, the latter only admit trivial ones. This extends results of Mart\'in, Mas and Vukoti\'c from the disc to the ball.
  Some of our results continue to hold when $d = \infty$.
\end{abstract}

\maketitle

\section{Introduction}

Broadly speaking, this article is concerned with symmetries of reproducing kernel Hilbert spaces (RKHS) of holomorphic
functions on the Euclidean open unit ball $\mathbb{B}_d$ in $\mathbb{C}^d$.
We study unitarily invariant spaces, which by definition are RKHS $\mathcal{H}$ whose reproducing kernel is of the form
\begin{equation*}
  K(z,w) = \sum_{n=0}^\infty a_n \langle z,w \rangle^n \quad (z,w \in \mathbb{B}_d),
\end{equation*}
where $(a_n)$ is a sequence of non-negative real numbers. We also assume that $a_0 = 1$ (hence $1 \in \mathcal{H}$ and $\|1\| = 1$)
and $a_1 > 0$ (which implies that $\mathcal{H}$ contains all coordinate functions $z_j$).
This class contains many important spaces such as Hardy, Bergman \cite{ZZ08,Zhu05}, Drury--Arveson \cite{Arveson98,Hartz23,Shalit15} and Dirichlet spaces \cite{EKM+14}.

It easily follows from the definition that if $f \in \mathcal{H}$ and $U$ is a unitary $d \times d$ matrix, then $f \circ U \in \mathcal{H}$
and $\|f \circ U \| = \|f\|$. Thus, every unitary $d \times d$ matrix induces
a unitary composition operator on $\mathcal{H}$; we regard these as symmetries of the RKHS.

This naturally raises the question: Does $\mathcal{H}$ admit other symmetries?
Being a Hilbert space, $\mathcal{H}$ admits a large number of unitary operators, but most of these will not
respect the function space structure of $\mathcal{H}$ in any reasonable way. In the the context of RKHS,
a natural and frequently studied class of operators preserving the function space structure are weighted composition
operators (WCO), which are operators of the form
\begin{equation*}
  W_{\delta,\phi}: \mathcal{H} \to \mathcal{H}, \quad f \mapsto \delta \cdot (f \circ \phi).
\end{equation*}
Here, $\delta: \mathbb{B}_d \to \mathbb{C}$ is a function and $\phi: \mathbb{B}_d \to \mathbb{B}_d$ is a mapping.
As special cases, weighted composition operators include composition operators ($\delta = 1$) and multiplication operators
($\phi = \id$). Some of the history of weighted composition operators is described in the introductions of \cite{Le12} and \cite{Zorboska18}. In addition, we mention that weighted composition operators appear in the study of homogeneous operator tuples (see for instance \cite{MN90,KM19,GKP22}) and as cyclicity preserving operators (see for instance \cite[Section 5]{Sampat24}).

We will characterize unitary and even co-isometric weighted composition operators on unitarily invariant spaces on $\mathbb{B}_d$.
If $d = 1$, i.e.\ for spaces on the disc, this was done by Mart\'in, Mas and Vukoti\'c \cite{MMV20}.
Our results generalize theirs, but our proofs are different.
Indeed, in several variables, we cannot rely on the factorization theory for functions in the Hardy space as in \cite{MMV20}.

 For $\gamma \in (0,\infty)$, let $\mathcal{H}_\gamma$ be the RKHS on $\mathbb{B}_d$
with kernel
\begin{equation*}
  K(z,w) = \frac{1}{( 1- \langle z,w \rangle)^\gamma }.
\end{equation*}
These spaces include standard weighted Dirichlet spaces ($0 < \gamma < 1$), the Drury--Arveson space
($\gamma = 1$), the Hardy space ($\gamma = d$) and standard weighted Bergman spaces ($\gamma > d$),
but for instance not the classical Dirichlet space.
Just as Mart\'in, Mas and Vukoti\'c, we observe a dichotomy between the spaces $\mathcal{H}_\gamma$, which admit many unitary WCOs,
and all other spaces, which only admit trivial unitary WCOs.

Let $\Aut(\mathbb{B}_d)$ denote the group of biholomorphic automorphisms of $\mathbb{B}_d$.
WCOs on the spaces $\mathcal{H}_\gamma$ were already studied by Trieu Le \cite{Le12}, who essentially proved
the following result.

\begin{theorem}[Le]
  \label{thm:H_gamma_intro}
  Let $d \in \mathbb{N}$, let $\gamma \in (0,\infty)$ and let $\phi: \mathbb{B}_d \to \mathbb{B}_d$ and $\delta: \mathbb{B}_d \to \mathbb{C}$.
  Then the following are equivalent:
  \begin{enumerate}[label=\normalfont{(\roman*)}]
    \item $W_{\delta,\phi}$ is a unitary operator on $\mathcal{H}_\gamma$;
    \item $W_{\delta,\phi}$ is a co-isometric operator on $\mathcal{H}_\gamma$;
    \item $\phi \in \Aut(\mathbb{B}_d)$ and $\delta(z) = \mu \frac{K(z,a)}{K(a,a)^{1/2}}$,
      where $a = \phi^{-1}(0)$ and $\mu \in \mathbb{C}$ with $|\mu| = 1$.
  \end{enumerate}
\end{theorem}
This is essentially \cite[Corollary 3.6]{Le12}; see also Section \ref{sec:H_gamma} below.

Our main result shows that all other spaces are very rigid in the sense that they only admit
trivial unitary WCOs.
\begin{theorem}
  \label{thm:non_H_gamma_intro}
  Let $d \in \mathbb{N}$ and let $\phi: \mathbb{B}_d \to \mathbb{B}_d$ and $\delta: \mathbb{B}_d \to \mathbb{C}$.
  Let $\mathcal{H}$ be a unitarily invariant space with $\mathcal{H} \neq \mathcal{H}_\gamma$ for all $\gamma \in (0,\infty)$.
  Then the following are equivalent:
  \begin{enumerate}[label=\normalfont{(\roman*)}]
    \item $W_{\delta,\phi}$ is a unitary operator on $\mathcal{H}$;
    \item $W_{\delta,\phi}$ is a co-isometric operator on $\mathcal{H}$;
    \item $\phi$ is a unitary $d \times d$ matrix and $\delta$ is a unimodular constant.
  \end{enumerate}
\end{theorem}
This theorem will be proved in Section \ref{sec:non-H_gamma}.

We will also study the case $d = \infty$, where $\mathbb{B}_d$ is interpreted as the open unit ball of $\ell^2$.
Our chief motivation is a theorem of Agler and \mcc,
which shows that the Drury--Arveson space on $\mathbb{B}_\infty$ (i.e.\ $\mathcal{H}_1$ in our notation) is a universal complete Pick space; see \cite{AM00} and \cite[Chapter 8]{AM02}.
When $d = \infty$, the function theory becomes more difficult, and so our results are less complete.
Nonetheless, the spaces $\mathcal{H}_\gamma$ are still tractable.
The fact that $\ell^2$ admits non-unitary isometries implies that not every co-isometric weighted composition operator on $\mathcal{H}_\gamma$ is unitary when $d = \infty$. Thus, Theorem \ref{thm:H_gamma_intro} now splits into two statements.
The equivalence between (i) and (iii) in Theorem \ref{thm:H_gamma_intro} continues to hold in case $d = \infty$.
On the other hand, co-isometric weighted composition operators can incorporate an additional linear isometry on $\ell^2$,
see Section \ref{sec:H_gamma} for details.

For non-$\mathcal{H}_\gamma$ spaces, we obtain the following rigidity result.

\begin{theorem}
  \label{thm:non_H_gamma_infinite_intro}
  Let $d = \infty$
  and let $\mathcal{H}$ be a unitarily invariant space with $\mathcal{H} \neq \mathcal{H}_\gamma$ for all $\gamma \in (0,\infty)$.
  Let $\phi: \mathbb{B}_\infty \to \mathbb{B}_\infty$ and $\delta: \mathbb{B}_\infty \to \mathbb{C}$ be such that $W_{\delta,\phi}$ is unitary
  on $\mathcal{H}$. If either $\phi \in \Aut(\mathbb{B}_\infty)$ or the reproducing kernel of $\mathcal{H}$ is bounded,
  then $\phi$ is a unitary.
\end{theorem}

This result is a combination of Proposition \ref{prop:non_h_gamma_auto} and Proposition \ref{prop:bounded_kernel}.
It remains open if Theorem \ref{thm:non_H_gamma_infinite_intro} holds for spaces with unbounded kernel
without the a priori assumption $\phi \in \Aut(\mathbb{B}_\infty)$.

\subsection*{Acknowledgements} The authors are grateful to Gadadhar Misra and Dmitry Yakubovich for pointing out the connection to homogeneous operator tuples and for providing relevant references. The authors are also very grateful to Frej Dahlin for alerting them to a gap in a previous version of the proof of Lemma \ref{lem:unbounded}.

\section{Preliminaries}

\subsection{Reproducing kernels}
A \emph{reproducing kernel Hilbert space} of functions on a set $X$ is a Hilbert space $\mathcal{H}$,
consisting of functions on $X$, such that for all $x \in X$, the point evaluation functional
\begin{equation*}
  \mathcal{H} \to \mathbb{C}, \quad f \mapsto f(x),
\end{equation*}
is bounded. The \emph{reproducing kernel} of $\mathcal{H}$ is the unique function
$K: X \times X \to \mathbb{C}$ with the property that $K_x := K(\cdot,x) \in \mathcal{H}$
and
\begin{equation*}
  \langle f, K_x \rangle = f(x)
\end{equation*}
for all $f \in \mathcal{H}, x \in X$.
For general background on RKHS, the reader is refered to \cite{PR16}.

Let $\mathcal{K}$ be another reproducing kernel Hilbert space on a set $Y$ with reproducing kernel
$L$. Given $\phi: Y \to X$ and $\delta: Y \to \mathbb{C}$, we consider the weighted composition
operator
\begin{equation*}
  W_{\delta,\phi}: \mathcal{H} \to \mathcal{K}, \quad f \mapsto \delta \cdot (f \circ \phi).
\end{equation*}
In general, this operator need not be well defined, since the function on the right need not belong to $\mathcal{K}$.
We are interested in the case when it is. In particular, when we say that $W_{\delta,\phi}$ is bounded/co-isometric/unitary, this in particular includes the assumption that $W_{\delta,\phi}$ is well defined.
We call $\delta$ the \emph{multiplication symbol} and $\phi$ the \emph{composition symbol} of $W_{\delta,\phi}$.

The following general result about WCOs in RKHS is well known. For completeness, and since we do not have a good reference
for the precise statement, we provide the short proof.

\begin{lemma}
  \label{lem:RKHS_WCO}
  In the setting above, the following statements hold:
  \begin{enumerate}[label=\normalfont{(\alph*)}]
    \item 
    If $W_{\delta,\phi}: \mathcal{H} \to \mathcal{K}$ is bounded, then
    \begin{equation*}
      W_{\delta,\phi}^* L_y = \overline{\delta(y)} K_{\phi(y)}
    \end{equation*}
    for all $y \in Y$.
  \item The operator $W_{\delta,\phi}: \mathcal{H} \to \mathcal{K}$ is a co-isometry
    if and only if
    \begin{equation}
      \label{eqn:co_iso}
      L(y_1,y_2) = \delta(y_1) \overline{\delta(y_2)} K(\phi(y_1),\phi(y_2))
    \end{equation}
    for all $y_1,y_2 \in Y$.
  \end{enumerate}
\end{lemma}

\begin{proof}
  (a) For all $f \in \mathcal{H}$, we have
  \begin{equation*}
    \langle f, W_{\delta,\phi}^* L_y \rangle = \langle \delta (f \circ \phi), L_y \rangle
    = \delta(y) f(\phi(y))
    = \langle f, \overline{\delta(y)} K_{\phi(y)} \rangle,
  \end{equation*}
  which gives (a).

  (b) Assume that $W_{\delta,\phi}$ is a co-isometry. Then using part (a)
  in the third step,
  \begin{align*}
    L(y_1,y_2) = \langle L_{y_2}, L_{y_1} \rangle
    = \langle W_{\delta,\phi}^* L_{y_2}, W_{\delta,\phi}^* L_{y_1} \rangle 
    &= \delta(y_1) \overline{\delta(y_2)}
    \langle K_{\phi(y_2)}, K_{\phi(y_1)} \rangle  \\
    &= \delta(y_1) \overline{\delta(y_2)}
    K(\phi(y_1),\phi(y_2)),
  \end{align*}
  hence \eqref{eqn:co_iso} holds.

  Conversely, if \eqref{eqn:co_iso} holds, then the computation above shows that
  \begin{equation*}
    \langle L_{y_2}, L_{y_1} \rangle = \langle \overline{\delta(y_2)} K_{\phi(y_2)},
  \overline{\delta(y_1)} K_{\phi(y_1)} \rangle 
  \end{equation*}
  for all $y_1,y_2 \in Y$.
  Since the linear span of the reproducing kernels is dense in $\mathcal{K}$,
  there exists a unique linear isometry $V: \mathcal{K} \to \mathcal{H}$ such that
  \begin{equation*}
    V L_{y} = \overline{\delta(y)} K_{\phi(y)}
  \end{equation*}
  for all $y \in Y$. A computation as in (a) then yields that
  \begin{equation*}
    V^* f = \delta (f \circ \phi)
  \end{equation*}
  for all $f \in \mathcal{H}$; whence $W_{\delta,\phi}: \mathcal{H} \to \mathcal{K}$ is a (well-defined) co-isometry.
 \end{proof}

If $W_{\delta_1,\phi_1}: \mathcal{H}_1 \to \mathcal{H}_2$ and $W_{\delta_2,\phi_2}: \mathcal{H}_2 \to \mathcal{H}_3$ are two bounded weighted composition operators, then for all $f \in \mathcal{H}_1$, we have
\begin{equation}
  \label{eqn:product_WCO}
  W_{\delta_2,\phi_2} W_{\delta_1,\phi_1} f =
  \delta_2 (\delta_1 \circ \phi_2) (f \circ \phi_1 \circ \phi_2),
\end{equation}
hence the product $W_{\delta_2,\phi_2} W_{\delta_1,\phi_1}$ is another weighted composition operator,
with multiplication symbol $\delta_2 (\delta_1 \circ \phi_2)$ and composition symbol $\phi_1 \circ \phi_2$.

As mentioned in the introduction, we will consider reproducing kernel Hilbert spaces
on the Euclidean unit ball $\mathbb{B}_d$ in $\mathbb{C}^d$ with reproducing kernel of the form
\begin{equation*}
  K(z,w) = \sum_{n=0}^\infty a_n \langle z,w \rangle^n,
\end{equation*}
where $a_0 = 1, a_1 > 0$ and $a_n \ge 0$ for all $n \ge 2$.
We also allow $d = \infty$, in which case $\mathbb{C}^\infty$ is understood as $\ell^2$.
While do not require this, another useful description is that the set of monomials
\begin{equation*}
  \{ z^\alpha: a_{|\alpha|} \neq 0 \}
\end{equation*}
forms an orthogonal basis of $\mathcal{H}$ with
\begin{equation*}
\|z^\alpha\|^2 = \frac{\alpha!}{a_{|\alpha|} |\alpha|!}
\end{equation*}
Here, $\alpha$ is a multi-index with finitely many non-zero entries.
See for instance \cite[Proposition 4.1]{GHX04}.

\subsection{Holomorphic functions}
Since we will also consider the case $d = \infty$, we briefly have to discuss infinite dimensional holomorphy.
Fortunately, we only require elementary parts of the theory.
More background information can for instance be found in \cite{Chae85,Dineen81,Mujica86}.
Let $\mathcal{E}$ and $\mathcal{F}$ be Hilbert spaces
and let $U \subset \mathcal{E}$ be open. We say that a function $f: U \to \mathcal{F}$ is holomorphic
if $f$ is locally bounded and for all $z \in U, \xi \in \mathcal{E}$ and $\eta \in \mathcal{F}$, the function
of one complex variable
\begin{equation*}
  \lambda \mapsto \langle f(z + \lambda \xi), \eta \rangle
\end{equation*}
is holomorphic in a neighborhood of $0$; see for instance \cite[Chapter 2]{Dineen81}, \cite[Section II.8]{Mujica86} or  \cite[Chapter 14]{Chae85} for how this definition is equivalent to various other ones.
In the finite dimensional setting, this reduces to the usual notion from several complex variables.

\begin{proposition}
  \label{prop:H_holo}
  Let $d \in \mathbb{N} \cup \{\infty\}$ and let $\mathcal{H}$
  be a unitarily invariant space on $\mathbb{B}_d$.
  Then every $f \in \mathcal{H}$ is holomorphic on $\mathbb{B}_d$.
\end{proposition}

\begin{proof}
  If $f \in \mathcal{H}$ and $r \in [0,1)$, then by the Cauchy--Schwarz inequality,
  \begin{equation}
    \label{eqn:CS_Hilbert_estimate}
    |f(z)| \le \|f\| K(z,z)^{1/2} \le \|f\| K(r e_1, r e_1)^{1/2} \quad \text{ for all } z \in r \mathbb{B}_d.
  \end{equation}
  Hence $f$ is locally bounded. Moreover, $K(\cdot,w)$ is holomorphic
  for each $w \in \mathbb{B}_d$.
  Inequality \eqref{eqn:CS_Hilbert_estimate} also shows that convergence in $\mathcal{H}$
  implies locally uniform convergence on $\mathbb{B}_d$.
  Since the linear span of the kernel functions is dense in $\mathcal{H}$,
  it follows that each $f \in \mathcal{H}$ is holomorphic.
\end{proof}

If $f: \mathbb{B}_d \to \mathcal{F}$ is holomorphic, then there exists a bounded linear operator
$Df (0): \mathbb{C}^d \to \mathcal{F}$ such that for all $\xi \in \mathbb{C}^d$ and $\eta \in \mathcal{F}$, we have
\begin{equation*}
  \frac{d}{d \lambda} \langle f( \lambda \xi), \eta \rangle \Big|_{\lambda = 0}
  = \langle D f(0) \xi, \eta \rangle.
\end{equation*}
The operator $D f(0)$ is the Fr\'echet derivative of $f$ at $0$, which always exists for holomorphic functions; see for instance \cite[Theorem 13.16]{Mujica86} or \cite[Theorem 14.13]{Chae85}.

We require the following version of the Schwarz lemma.

\begin{theorem}
  \label{thm:Schwarz}
  Let $f: \mathbb{B}_d \to \mathbb{B}_{d'}$ be holomorphic with $f(0) = 0$. Then $Df(0)$ has operator
  norm at most $1$. If $Df(0)$ is an isometry, then $f(z) = Df(0) z$ for all $z \in \mathbb{B}_d$.
\end{theorem}

\begin{proof}
  The proof in Rudin's book in the finite dimensional setting \cite[Chapter 8]{Rudin08} carries
  over almost verbatim. See also \cite{Harris69}, \cite[Theorem 4.3]{Dineen81} or \cite[Theorem 13.13]{Chae85} for more general versions of the Schwarz lemma
  in the infinite dimensional setting.
\end{proof}

We also need some basic facts about automorphisms of $\mathbb{B}_d$.
A treatment in the finite dimensional setting can be found in Rudin's book \cite[Chapter 2]{Rudin08};
the facts we need carry over to $d = \infty$, see also \cite{HS71} for $d = \infty$.

A biholomorphic
automorphism of $\mathbb{B}_d$ is a bijective holomorphic map $\phi: \mathbb{B}_d \to \mathbb{B}_d$
such that the inverse $\phi^{-1}$ is also holomorphic.
We write $\Aut(\mathbb{B}_d)$ for the group of biholomorphic automorphisms of $\mathbb{B}_d$.
Every unitary on $\mathbb{C}^d$ yields an element of $\Aut(\mathbb{B}_d)$.
Moreover, for each $a \in \mathbb{B}_d$, there is an involution $\varphi_a \in \Aut(\mathbb{B}_d)$
mapping $0$ to $a$. If $P_a$ denotes the orthogonal projection onto the linear span of $a$,
$Q_a = I - P_a$ and $s_a = (1 - \|a\|^2)^{1/2}$, then $\varphi_a$ is given by
\begin{equation*}
  \varphi_a(z) = \frac{a - P_a z - s_a Q_a z}{1 - \langle z,a \rangle }.
\end{equation*}
With these definitions, we have the following well-known description of $\Aut(\mathbb{B}_d)$.
Indeed, with the automorphisms $\varphi_a$ in hand, this can be deduced from the Schwarz lemma (Theorem \ref{thm:Schwarz}).

\begin{theorem}
  \label{thm:auto_group}
  Let $d \in \mathbb{N} \cup \{\infty\}$. Then
  \begin{equation*}
    \Aut(\mathbb{B}_d) = \{ U \circ \varphi_a: U \text{ unitary }, a \in \mathbb{B}_d \}.
  \end{equation*}
\end{theorem}

\begin{remark}
  If $V$ is an isometry, then a computation shows that
  \begin{equation}
    \label{eqn:involution_iso}
    V \varphi_a = \varphi_{V a} V.
  \end{equation}
  Thus, we also have the reverse description
  \begin{equation*}
    \Aut(\mathbb{B}_d) = \{\varphi_a \circ U: U \text{ unitary }, a \in \mathbb{B}_d\}.
  \end{equation*}
  However, if we consider not necessarily surjective maps in case $d = \infty$, then the order 
  matters. From \eqref{eqn:involution_iso}, we see that
  \begin{equation*}
    \{ \varphi_a \circ V: V \text{ isometry},  a \in \mathbb{B}_d \} \supset 
    \{ V \circ \varphi_a: V \text{ isometry }, a \in \mathbb{B}_d \},
  \end{equation*}
  but the inclusion is strict if $d = \infty$. Indeed, the range
  of every map on the right contains $0$, but if $a$ does not belong to the range of $V$,
  then $0$ does not belong to the range of $\varphi_a \circ V$.

  For our purposes, the larger set will be relevant, and this set can alternatively be described as
  \begin{equation*}
    \{ \varphi_a \circ V: V \text{ isometry},  a \in \mathbb{B}_d \} =
    \{ \phi \circ V: V \text{ isometry},  \phi \in \Aut(\mathbb{B}_d) \}.
  \end{equation*}
\end{remark}

\section{Weighted composition operators on the ball}

In this section, we specialize some of the general results for weighted composition operators
to unitarily invariant spaces on the ball.
Throughout, let $d \in \mathbb{N} \cup \{\infty\}$ and let $\mathcal{H}$ be a unitarily
invariant space on $\mathbb{B}_d$ with reproducing kernel $K$.

\begin{lemma}
  \label{lem:WCO_ball}
  Let $\phi: \mathbb{B}_d \to \mathbb{B}_d$ and $\delta: \mathbb{B}_d \to \mathbb{C}$
  be mappings such that $W_{\delta,\phi}$ is co-isometric on $\mathcal{H}$.
  Let $a = \phi(0)$.
  Then:
  \begin{enumerate}[label=\normalfont{(\alph*)}]
    \item $K(\phi(\cdot),a)$ does not vanish and there exists $\mu$ in $\mathbb{C}$ with $|\mu|=1$ such that
  \begin{equation*}
    \delta(z) = \mu \frac{K(a,a)^{1/2}}{K(\phi(z),a)} \quad \text{ for all } z \in \mathbb{B}_d;
  \end{equation*}
  \item
  \begin{equation*}
    K(z,w) = \frac{K(\phi(z),\phi(w))K(a,a)}{K(\phi(z),a) K(a,\phi(w))} \quad \text{ for all } z,w \in \mathbb{B}_d;
  \end{equation*}
    \item $\delta$ and $\phi$ are holomorphic;
    \item $\phi$ is injective.
  \end{enumerate}
\end{lemma}

\begin{proof}

  (a) and (b)
  Lemma \ref{lem:RKHS_WCO} implies that that
  \begin{equation}
    \label{eqn:WCO_ball}
    K(z,w) = \delta(z) \overline{\delta(w)} K(\phi(z),\phi(w)) \quad \text{ for all } z,w \in \mathbb{B}_d.
  \end{equation}
  Setting $w = 0$, we obtain
  \begin{equation*}
    1 = \delta(z) \overline{\delta(0)} K(\phi(z),a).
  \end{equation*}
  Hence $K(\phi(\cdot),a)$ does not vanish. Setting $z=0$ gives
  $|\delta(0)|^2 K(a,a) = 1$, so we obtain (a).
  Substitung this formula for $\delta$ in \eqref{eqn:WCO_ball} gives (b).

  (c)
  We have
  \begin{equation*}
    \mathcal{H} \ni W_{\delta,\phi} 1 = \delta,
  \end{equation*}
  hence $\delta$ is holomorphic by Proposition \ref{prop:H_holo}.
  Let $\eta \in \mathbb{C}^d$.
  The assumption that the first coefficient of the kernel satsfies $a_1 > 0$ implies
  that $\langle \cdot,\eta \rangle \in \mathcal{H}$. Hence
  \begin{equation*}
    \mathcal{H} \in W_{\delta,\phi} \langle \cdot, \eta \rangle 
    = \delta \langle \phi(\cdot), \eta \rangle.
  \end{equation*}
  Since $\delta$ does not vanish, this implies that $\phi$ is holomorphic as well.

  (d) Let $\phi(z) = \phi(w)$. Then the formula in (b) shows that $K(x,z) = K(x,w)$ for all $x \in \mathbb{B}_d$,
  which implies $z = w$ since $\mathcal{H}$ contains $\langle \cdot,\eta \rangle$ for all $\eta \in \mathbb{C}^d$ and
  hence separates the points of $\mathbb{B}_d$.
\end{proof}

The Schwarz lemma implies the following result.
\begin{lemma}
  \label{lem:Schwarz_kernel}
  Let $\phi: \mathbb{B}_d \to \mathbb{B}_d$ be holomorphic.
  If  
  \begin{equation*}
    K(\phi(z), \phi(w)) = K(z,w) \quad \text{ for all } z \in \mathbb{B}_d,
  \end{equation*}
  then there exists
  a linear isometry $V$ such that
  \begin{equation*}
    \phi(z) = V z \quad \text{ for all } z \in \mathbb{B}_d.
  \end{equation*}
\end{lemma}

\begin{proof}
  Let
  \begin{equation*}
    K(z,w) = \sum_{n=0}^\infty a_n \langle z,w \rangle^n.
  \end{equation*}
  The assumption implies in particular that $K(\phi(0), \phi(0)) = K(0,0) = 1$, so since $a_1 > 0$,
  this forces $\phi(0) = 0$.
  Let $V = D\phi(0)$. We show that $V$ is an isometry; the result then follows from Theorem \ref{thm:Schwarz}, the variant of the Schwarz lemma.

  To see that $V$ is an isometry, fix $z,w \in \mathbb{B}_d$. Then for all $\lambda,\mu \in \mathbb{D}$, we have
  \begin{equation*}
    K(\phi(\lambda z), \phi(\mu w)) = K(\lambda z, \mu w).
  \end{equation*}
  Differentiating both sides with respect to $\lambda$ and evaluating at $\lambda = 0$, we find that
  \begin{equation*}
    a_1 \langle D \phi(0) z, \phi(\mu w) \rangle = a_1 \langle z, \mu w \rangle.
  \end{equation*}
  Since $a_1 > 0$, this implies
  \begin{equation*}
    \langle \phi(\mu w), D \phi(0) z \rangle  = \langle \mu w, z \rangle. 
  \end{equation*}
  Taking the derivative with respect to $\mu$ at the origin yields
  \begin{equation*}
    \langle D \phi(0) w, D \phi(0) z \rangle = \langle w,z \rangle,
  \end{equation*}
  so that $V = D \phi(0)$ is an isometry, as desired.
\end{proof}

We also need the following simple observation.

\begin{lemma}
  \label{lem:WCO_with_iso}
  Let $V$ be a linear isometry and let $\mu \in \mathbb{C}$ with $|\mu| = 1$.
  Then $W_{\mu,V}$ is a co-isometry. Moreover, $W_{\mu,V}$ is unitary if and only if $V$ is unitary.
\end{lemma}

\begin{proof}
  It is immediate from Lemma \ref{lem:RKHS_WCO} that $W_{\mu,V}$ is a co-isometry.
  If $V$ is unitary, then $W_{\mu,V}$ is invertible with inverse $W_{\overline{\mu},V^*}$, hence $W_{\mu,V}$ is unitary.

  Conversely, assume that $W_{\mu,V}$ is unitary and let $z \in \mathbb{B}_d$ with $V^* z = 0$.
  Then
  \begin{equation*}
    W_{\mu,V} K_z = \mu K_z \circ V = \mu K_{V^* z} = \mu.
  \end{equation*}
  Since $W_{\mu,V}$ is an isometry, it follows that $1 = \|K_z\|^2 = K(z,z)$.
  Since the first coefficient of $K$ satisfies $a_1 > 0$, this implies $z= 0$, showing
  that $V^*$ is injective, i.e. $V$ is unitary.
\end{proof}

\section{$\mathcal H_\gamma$ spaces}
\label{sec:H_gamma}

In this section, we study the spaces $\mathcal{H}_\gamma$ with reproducing kernel
\begin{equation*}
  K(z,w) = \frac{1}{(1 - \langle z,w \rangle)^\gamma},
\end{equation*}
and generalize Le's theorem to infinite dimensions.
As mentioned in the introduction, co-isometric weighted composition operators need not be unitary in infinite
dimensions, so Le's theorem splits into two parts. We first consider the co-isometric case.

\begin{theorem}
  \label{thm:H_gamma_co_iso}
  Let $d \in \mathbb{N} \cup \{\infty\}$, let $\gamma \in (0,\infty)$ and
  let $\phi: \mathbb{B}_d \to \mathbb{B}_d$ and $\delta: \mathbb{B}_d \to \mathbb{C}$.
  The following are equivalent:
  \begin{enumerate}[label=\normalfont{(\roman*)}]
    \item $W_{\delta,\phi}$ is a co-isometry on $\mathcal{H}_\gamma$;
    \item there exist $a \in \mathbb{B}_d$ and an isometry $V$ on $\mathbb{C}^d$ such that $\phi = \varphi_a V$
      and $\delta(z) = \mu \frac{K(Vz,a)}{K(a,a)^{1/2}}$
      for some $\mu \in \mathbb{C}$ with $|\mu| = 1$.
  \end{enumerate}
\end{theorem}

\begin{proof}
  (ii) $\Rightarrow$ (i)
  This proof proceeds essentially as Le's proof in the finite dimensional setting; see \cite[Proposition 3.1]{Le12}.
  The key point is the formula
  \begin{equation}
    \label{eqn:auto_formula}
    1 - \langle \varphi_a(z), \varphi_a(w) \rangle = \frac{(1 - \|a\|^2) (1 - \langle z,w \rangle) }{(1 - \langle z,a \rangle)(1 - \langle a,w \rangle)  };
  \end{equation}
  see \cite[Theorem 2.2.2]{Rudin08}.
  From this formula, it follows that
  \begin{equation*}
    K(z,w) = \frac{K(z,a) K(a,w)}{K(a,a)} K(\varphi_a(z), \varphi_a(w)).
  \end{equation*}
  (Some care must be taken when raising \eqref{eqn:auto_formula} to the power of $\gamma$.
  But each factor appearing in \eqref{eqn:auto_formula} has positive real part, and the identity $(\lambda \mu)^\gamma = \lambda^\gamma \mu^\gamma$
  holds for complex numbers $\lambda,\mu$ with positive real part, so this operation is valid after
  we arrange \eqref{eqn:auto_formula} to have two complex factors on each side.)
  Replacing $z$ with $V z$ and $w$ with $V w$, we find that
  \begin{equation}
    \label{eqn:K_V_phi}
    K(z,w) = \frac{K(Vz,a) K(a,Vw)}{K(a,a)} K(\varphi_a(V z), \varphi_a(V w)).
  \end{equation}
  Recalling the definition of $\delta$ and $\phi$ in the statement of the theorem,
  Lemma \ref{lem:RKHS_WCO} then shows that $W_{\delta,\phi}$ is a co-isometry.

  (i) $\Rightarrow$ (ii)
  While Le's proof can be generalized to infinite dimensions, we argue slightly differently.
  We first consider the case when $\phi(0) = 0$.
  Lemma \ref{lem:WCO_ball} shows that $\delta$ is a unimodular constant, $\phi$ is holomorphic,
  and
  \begin{equation*}
    K(z,w) = K(\phi(z),\phi(w)) \quad \text{ for all } z,w \in \mathbb{B}_d.
  \end{equation*}
  Lemma \ref{lem:Schwarz_kernel} now yields an isometry $V$ with $\phi = V$. This completes the proof
  in case $\phi(0) = 0$.

  Next, let $\phi: \mathbb{B}_d \to \mathbb{B}_d$ be an arbitrary mapping. Let $a = \phi(0)$ and $\psi(z) = K(z,a) K(a,a)^{-1/2}$. By the already established implication, $W_{\psi,\varphi_a}$ is a co-isometry, hence so is the product
  $W_{\delta,\phi} W_{\psi, \varphi_a}$.
  By Equation \eqref{eqn:product_WCO}, this product is a weighted composition operator
  with composition symbol $\varphi_a \circ \phi$.
  Since $(\varphi_a \circ \phi)(0) = 0$, we may apply the already established case $a = 0$ to this weighted
  composition operator to obtain an isometry $V$ and a unimodular constant $\mu$ with
  $\varphi_a \circ \phi = V$ and
  \begin{equation*}
    \delta (\psi \circ \phi) = \mu.
  \end{equation*}
  Since $\varphi_a$ is an involution, $\phi = \varphi_a V$. Moreover,
  \begin{equation*}
    \delta(z) = \frac{\mu}{\psi(\phi(z))}
    = \mu \frac{K(a,a)^{1/2}}{K( \varphi_a( V z), a)}
    = \mu \frac{K(V z, a)}{K(a,a)^{1/2}},
  \end{equation*}
  where in the last step, we have used \eqref{eqn:K_V_phi} with $w = 0$.
\end{proof}

The case of unitary weighted composition operators is covered in the following result.

\begin{corollary}
  \label{cor:H_gamma_unitary}
  Let $d \in \mathbb{N} \cup \{\infty\}$, let $\gamma \in (0,\infty)$ and let $\phi: \mathbb{B}_d \to \mathbb{B}_d$
  and $\delta: \mathbb{B}_d \to \mathbb{C}$. The following are equivalent:
  \begin{enumerate}[label=\normalfont{(\roman*)}]
    \item $W_{\delta,\phi}$ is a unitary on $\mathcal{H}_\gamma$;
    \item $\phi \in \Aut(\mathbb{B}_d)$ and $\delta(z) = \mu \frac{K(z,a)}{K(a,a)^{1/2}}$,
      where $a = \phi^{-1}(0)$ and $\mu \in \mathbb{C}$ with $|\mu| = 1$.
  \end{enumerate}
\end{corollary}

\begin{proof}
  (ii) $\Rightarrow$ (i)
  By Theorem \ref{thm:auto_group}, there exists a unitary $U$
  such that $\phi = U \varphi_a$, where necessarily $a = \phi^{-1}(0)$. By Theorem \ref{thm:H_gamma_co_iso},
  $W_{\delta,\varphi_a}$ is a co-isometry.
  The operator $W_{\delta,\varphi_a}$ is injective since $\delta$ does not vanish and $\varphi_a$
  is surjective. Hence $W_{\delta,\varphi_a}$ is unitary. It is clear that $W_{1,U}$ is unitary,
  hence so is $W_{\delta,\phi} = W_{\delta,\phi_a} W_{1,U}$.

  (i) $\Rightarrow$ (ii)
  By Theorem \ref{thm:H_gamma_co_iso}, there exist $b \in \mathbb{B}_d$ and an isometry
  $V$ such that $\phi = \varphi_b V$ and $\delta(z) = \mu \frac{K(V z,b)}{K(b,b)^{1/2}}$, where $|\mu| = 1$.
  We will show that $V$ is unitary.

  If $b = 0$, Lemma \ref{lem:WCO_with_iso} shows that $V$ is unitary.
  If $b$ is arbitary, as in the proof of Theorem \ref{thm:H_gamma_co_iso}, we use the operator $W_{\psi,\varphi_b}$, where $\psi(z) = \frac{K(z,b)}{K(b,b)^{1/2}}$. This operator is unitary by the already established implication,
  so $W_{\delta,\phi} W_{\psi,\varphi_b}$ is a unitary weighted composition operator with composition
  symbol $\varphi_b \circ \phi = V$. By the preceding paragraph, $V$ is unitary.

  Since $V$ is unitary, $\phi = \varphi_b V \in \Aut(\mathbb{B}_d)$.
  Let $a = V^* b = \phi^{-1}(0)$. Then
  \begin{equation*}
    \delta(z) = \mu \frac{K(V z, b)}{K(b,b)^{1/2}} = \mu \frac{K(z,a)}{K(a,a)^{1/2}},
  \end{equation*}
  as desired.
\end{proof}

\section{Non-$\mathcal H_\gamma$ spaces}

\label{sec:non-H_gamma}

In this section, we consider unitarily invariant spaces that are not one of the spaces $\mathcal{H}_\gamma$.
Our goal is to show that these spaces only admit trivial co-isometric weighted composition operators.
In particular, this will establish the dichotomy between the spaces $\mathcal{H}_\gamma$ and all other spaces.

We first study weighted composition operators whose composition symbol is an automorphism.
In this setting, we will make use of the fact that the unitary group is a maximal subsemigroup of $\Aut(\mathbb{B}_d)$.

\begin{proposition}
  \label{prop:max_subsemigroup}
  Let $d \in \mathbb{N} \cup \{\infty\}$ and let $S \subset \Aut(\mathbb{B}_d)$ be a subsemigroup that
  properly contains the unitary group. Then $S = \Aut(\mathbb{B}_d)$.
\end{proposition}

\begin{proof}
  This is \cite[Proposition 9.9]{Hartz17a}, where the result was proved using the disc trick of Davidson, Ramsey
  and Shalit \cite{DRS11}. While the result in \cite{Hartz17a} is stated for $d < \infty$, the proof extends
  essentially verbatim to $d = \infty$.
  Indeed, the proof in \cite{Hartz17a} proceeds by a reduction to the one dimensional case with the help of the description of $\Aut(\mathbb{B}_d)$ of Theorem \ref{thm:auto_group}. (The disc trick is also directly available if $d = \infty$; see \cite[Section 4.2]{HS22}.)

  If we assume that $S$ is a closed subgroup of $\Aut(\mathbb{B}_d)$, then the result also follows
  from an earlier and very general theorem of Kaup and Upmeier \cite{KU76}; see also \cite[Theorem 1.2]{BKU78}.
\end{proof}

The result regarding weighted composition operators whose composition symbol is an automorphism is the following.

\begin{proposition}
  \label{prop:non_h_gamma_auto}
  Let $d \in \mathbb{N} \cup \{\infty\}$ and let $\mathcal{H}$ be a unitarily invariant space
  on $\mathbb{B}_d$ that is not one of the spaces $\mathcal{H}_\gamma$.
  If $\phi \in \Aut(\mathbb{B}_d)$ and $\delta: \mathbb{B}_d \to \mathbb{C}$ are such that
  $W_{\delta,\phi}$ is co-isometric on $\mathcal{H}$, then $\phi$ is unitary.
\end{proposition}

\begin{proof}
  Suppose that $\phi$ is not given by a unitary. We will show that $\mathcal{H} = \mathcal{H}_\gamma$
  for some $\gamma \in (0,\infty)$.
  To this end, let
  \begin{equation*}
    S = \{ \psi \in \Aut(\mathbb{B}_d): \exists \tau \text{ such that } W_{\tau,\psi} \text{ is co-isometric} \}.
  \end{equation*}
  Since $\mathcal{H}$ is unitarily invariant, $S$ contains all unitaries.
  Equation \eqref{eqn:product_WCO} shows that $S$ is a subsemigroup of $\Aut(\mathbb{B}_d)$.
  By assumption, $\phi \in S$, and $\phi$ is not a unitary. Hence $S = \Aut(\mathbb{B}_d)$ by Proposition \ref{prop:max_subsemigroup}.

  Now, Lemma \ref{lem:WCO_ball} shows that
  \begin{equation*}
    K(z,w) = \frac{K(\phi(z),\phi(w)) K(\phi(0),\phi(0))}{K(\phi(z),\phi(0)) K(\phi(0),\phi(w))}
  \end{equation*}
  for all $z,w \in \mathbb{B}_d$ and all $\phi \in \Aut(\mathbb{B}_d)$,
  and the denominator does not vanish.
  Replacing $z$ with $\phi(z)$, $w$ with $\phi(w)$ and $\phi$ with $\phi^{-1}$, we find that
  \begin{equation*}
    K(\phi(z),\phi(w)) = \frac{K(z,w) K(a,a)}{K(z,a) K(a,w)}
  \end{equation*}
  for all $z,w \in \mathbb{B}_d$ and $\phi \in \Aut(\mathbb{B}_d)$,
  where $a = \phi^{-1}(0)$. This is precisely the setting of \cite[Proposition 4.3]{Hartz17a},
  which implies that $\mathcal{H} = \mathcal{H}_\gamma$ for some $\gamma \in (0,\infty)$.
\end{proof}

Proposition \ref{prop:non_h_gamma_auto} shows that non-$\mathcal{H}_\gamma$ spaces do not admit
any non-trivial co-isometric weighted composition operators, provided we assume a priori that the composition
symbol is an automorphism.
To remove this a priori assumption, similar to Mart\'in, Mas and Vukoti\'c \cite{MMV20}, we distinguish two cases,
namely that of a bounded reproducing kernel and that of an unbounded reproducing kernel.
Observe that the $\mathcal{H}_\gamma$ spaces fall into the second category.

We first consider the case of bounded kernel. The following result generalizes \cite[Theorem 4]{MMV20}
from the disc to the unit ball, allowing $d = \infty$.
The basic idea of the proof is the same, but the details differ slightly since we cannot use Hardy space theory
such as the existence of radial limits.

\begin{proposition}
  \label{prop:bounded_kernel}
  Let $d \in \mathbb{N} \cup \{ \infty \}$ and let $\mathcal{H}$ be a unitarily invariant
  space with reproducing kernel $K$. Assume that $\sup_{z \in \mathbb{B}_d} K(z,z) < \infty$.
  If $\phi: \mathbb{B}_d \to \mathbb{B}_d$ and $\delta: \mathbb{B}_d \to \mathbb{C}$ are such that
  $W_{\delta,\phi}$ is co-isometric on $\mathcal{H}$, then $\phi$ is a linear isometry.
  If $W_{\delta,\phi}$ is a unitary on $\mathcal{H}$, then $\phi$ is a unitary on $\mathbb{C}^d$.
\end{proposition}

\begin{proof}
  Let $M = \sup_{z \in \mathbb{B}_d} K(z,z) = \sum_{n=0}^\infty a_n$.
  Since $M < \infty$, the series defining $K$ converges uniformly on $\overline{\mathbb{B}_d} \times \overline{\mathbb{B}_d}$ and defines a jointly (norm) continuous function there. Thus, the elements of $\mathcal{H}$ all extend
  to (norm) continuous functions on the closed ball $\overline{\mathbb{B}_d}$.
  In particular, this is true for $\delta = W_{\delta,\phi} 1 \in \mathcal{H}$.

  Assume that $W_{\delta,\phi}$ is co-isometric.
  Lemma \ref{lem:RKHS_WCO} shows that
  \begin{equation}
    \label{eqn:bounded_kernel}
    K(z,w) = \delta(z) \overline{\delta(w)} K(\phi(z),\phi(w)) \quad \text{ for all } z,w \in \mathbb{B}_d.
  \end{equation}
  Taking $z=w$, we find that
  \begin{equation*}
    |\delta(z)|^2 = \frac{K(z,z)}{K(\phi(z),\phi(z))} \ge \frac{K(z,z)}{M} \quad \text{ for all } z \in \mathbb{B}_d.
  \end{equation*}
  From this formula, we deduce three things. First, $|\delta|^2 \ge \frac{1}{M}$ in $\overline{\mathbb{B}_d}$,
  second $|\delta| \ge 1$ on $\partial \mathbb{B}_d$, and third $|\delta(0)| \le 1$.
  Since $\delta$ is holomorphic by Lemma \ref{lem:WCO_ball}, we may apply the maximumum modulus principle to $\frac{1}{\delta}$ on each disc $\overline{\mathbb{D}} \zeta$
  for $\zeta \in \partial \mathbb{B}_d$ to conclude that $\delta$ is a unimodular constant.

  Feeding this information into \eqref{eqn:bounded_kernel}, we obtain $K(z,w) = K(\phi(z),\phi(w))$
  for all $z,w \in \mathbb{B}_d$. Since $\phi$ is holomorphic (again by Lemma \ref{lem:WCO_ball}),
  Lemma \ref{lem:Schwarz_kernel} implies that $\phi$ is a linear isometry.

  If $W_{\delta,\phi}$ is in addition unitary, then Lemma \ref{lem:WCO_with_iso} implies that $\phi$ is unitary.
\end{proof}

The case of an unbounded kernel calls for a different argument.
Since we know that the $\mathcal{H}_\gamma$ spaces admit non-trivial co-isometric
weighted composition operators, the conclusion also cannot be that the composition symbol
is linear. It is here where we have to assume that $d < \infty$.

The following result generalizes \cite[Theorem 5]{MMV20} from the disc to the ball.
The argument of Mart\'in, Mas and Vukoti\'c in one variable uses results such as the inner/outer factorization
in the Hardy space, wich are not available in several variables.
The argument below is entirely different.
The authors thank Frej Dahlin for pointing out a gap in a previous version of the proof.

\begin{lemma}
  \label{lem:unbounded}
  Let $d \in \mathbb{N}$ and let $\mathcal{H}$ be a unitarily invariant
  space with reproducing kernel $K$. Assume that $\sup_{z \in \mathbb{B}_d} K(z,z) = \infty$.
  If $\phi: \mathbb{B}_d \to \mathbb{B}_d$ and $\delta: \mathbb{B}_d \to \mathbb{C}$ are such that
  $W_{\delta,\phi}$ is co-isometric on $\mathcal{H}$, then $\phi \in \Aut(\mathbb{B}_d)$.
\end{lemma}

\begin{proof}
  Let $a = \phi(0)$.
  Lemma \ref{lem:WCO_ball} shows that $\phi$ is holomorphic, injective, $K(\phi(z),a) \neq 0$ for all $z \in \mathbb{B}_d$ and
  \begin{equation}
    \label{eqn:unbounded_ratio}
    K(z,z) = \frac{K(\phi(z),\phi(z)) K(a,a)}{K(\phi(z),a) K(a,\phi(z))} \quad \text{ for all } z \in \mathbb{B}_d.
  \end{equation}
  Let $G = \phi(\mathbb{B}_d)$. Since $\phi$ is a holomorphic injection, $G$ is open and $\phi: \mathbb{B}_d \to G$ is a biholomorphism; see for instance 
  \cite[Theorem I.2.14]{Range86}.
It suffices to show that $\partial G \cap \mathbb{B}_d = \emptyset$. It then follows that
$G$ is relatively closed in $\mathbb{B}_d$, and hence $G = \mathbb{B}_d$ by connectedness of $\mathbb{B}_d$.

Suppose that $w \in \partial G \cap \mathbb{B}_d$. Then there exists a sequence $(z_n)$ in $\mathbb{B}_d$ such that
$z_n \to z \in \overline{\mathbb{B}_d}$ and $\phi(z_n) \to w$. Since $G$ is open, $w \notin G$, so $z \in \partial \mathbb{B}_d$.
Since $K(z_n,z_n) \to \infty$, it follows from \eqref{eqn:unbounded_ratio} that $K(w,a) = 0$.
If $K$ does not have any zeros, then we obtain a contradiction right away.

Otherwise, we argue as follows.
Recall that $K(z,w) = h(\langle z,w \rangle)$, where $h(t) = \sum_{n=0}^\infty a_n t^n$.
 Since $K(w,a) = 0$, we have $a \neq 0$.
The holomorphic function $h$ on $\mathbb{D}$ can only have finitely many zeros in $\{t \in \mathbb{D}: |t| \le \|a\| \}$,
say $\lambda_1,\ldots,\lambda_k$.
For $j=1,\ldots,k$, let
\begin{equation*}
  A_j = \{z \in \mathbb{C}^d: \langle z,a \rangle = \lambda_j \}.
\end{equation*}
Thus, $A_1,\ldots,A_k$ are distinct parallel complex hyperplanes.
Since $K(w,a) = 0$, we have $w \in A_j$ for some $j$.

This argument shows that
\begin{equation*}
  \partial G \cap \mathbb{B}_d \subset \bigcup_{j=1}^k A_j.
\end{equation*}
On the other hand, since $K(\phi(z),a) \neq 0$ for all $z \in \mathbb{B}_d$, we see that
\begin{equation*}
  G \cap \Big( \bigcup_{j=1}^k A_j \Big) = \emptyset.
\end{equation*}
In particular, if $w \in \partial G \cap \mathbb{B}_d$, then there exists an open ball $B$ centered at $w$
such that $\partial G \cap B \subset A_j$ for some $j$, while $G \cap A_j = \emptyset$.
This contradicts the lemma below, which then completes the proof of $\partial G \cap \mathbb{B}_d = \emptyset$.
\end{proof}

We now come to the missing lemma.

\begin{lemma}
  Let $\phi: \mathbb{B}_d \to G \subset \mathbb{C}^d$ be a biholomorphism,
  let $B \subset \mathbb{C}^d$ be an open ball and let
  $A \subset \mathbb{C}^d$ be complex hyperplane.
  If $\partial G \cap B \subset A$ and $G \cap A = \emptyset$, then $\partial G \cap B = \emptyset$.
\end{lemma}

\begin{proof}
  Suppose that $w \in \partial G \cap B$. By applying an affine linear map,
  we may assume that $w = 0$ and that $A = \{0\} \times \mathbb{C}^{d-1}$.
  Since $0 \in \partial G$, there exist an open disc $D$ centered at $0$ and an open ball $B' \subset \mathbb{C}^{d-1}$
  centered at $0$ such that $D  \times B' \subset B$
  and $(D \times B') \cap G \neq \emptyset$.
  The set $U:=(D \setminus \{0\}) \times B'$ is connected
  as a Cartesian product of connected sets. Since $G \cap A = \emptyset$,
  we have $U \cap G \neq \emptyset$,
  and since $\partial G \cap B \subset A$, we have $U \cap \partial G = \emptyset$.
  Hence $U \subset G$ by connectedness of $U$; in particular, $(D \setminus \{0\}) \times \{0\} \subset G$.

  Now the Riemann removable singularities theorem for functions of one variable implies that
  \begin{equation*}
    D \setminus \{0\} \to \mathbb{B}_d, \quad z \mapsto \phi^{-1}(z,0),
  \end{equation*}
  extends to a holomorphic function $h: D \to \overline{\mathbb{B}_d}$.
  The maximum modulus principle shows that $\|h(0)\| \le \sup_{z \in \frac{1}{2} \partial D} \|h(z)\| < 1$,
  so $h$ takes values in $\mathbb{B}_d$.
  Then $(\phi \circ h) (z) = (z,0)$ for all $z \in D \setminus \{0\}$, and
  hence for all $z \in D$ by continuity. But this would mean that $0 \in G$,
  contradicting the fact that $0 \in \partial G$ and $G$ is open.
\end{proof}

We can now prove Theorem \ref{thm:non_H_gamma_intro}.

\begin{proof}[Proof of Theorem \ref{thm:non_H_gamma_intro}]
  (i) $\Rightarrow$ (ii) is trivial.

  (ii) $\Rightarrow$ (iii) If $W_{\delta,\phi}$ is a co-isometry, then Lemma \ref{lem:unbounded} and Proposition \ref{prop:non_h_gamma_auto} imply that $\phi$ is unitary. It then follows from Lemma \ref{lem:WCO_ball} (a) that $\delta$ is a unimodular constant.

  (iii) $\Rightarrow$ (i) If $\phi$ is unitary and $\delta$ is a unimodular constant, then Lemma \ref{lem:RKHS_WCO} implies that $W_{\delta,\phi}$ is unitary.
\end{proof}

We do not know if the natural analogue of Lemma \ref{lem:unbounded} continues to hold if $d = \infty$,
so we state this as a question.

\begin{question}
  Let $d = \infty$ and let $\mathcal{H}$ be a unitarily invariant space with reproducing kernel $K$
  with $\sup_{z \in \mathbb{B}_\infty} K(z,z) = \infty$. Let $\phi: \mathbb{B}_d \to \mathbb{B}_d$
  and $\delta: \mathbb{B}_d \to \mathbb{C}$ be such that $W_{\delta,\phi}$ is co-isometric on $\mathcal{H}$.
  Does it follow that $ \phi = \varphi_a V$ for some $a \in \mathbb{B}_d$ and some linear isometry $V$?
\end{question}

Note that this question has a positive answer for the spaces $\mathcal{H}_\gamma$ (by Theorem \ref{thm:H_gamma_co_iso}).
It also has a positive answer if $\phi(0) = 0$ (by combining Lemmas \ref{lem:WCO_ball} and \ref{lem:Schwarz_kernel}).

\begin{remark}
  If $d = \infty$, then the proof strategy of Lemma \ref{lem:unbounded} fails.
  Firstly, the image of an injective holomorphic map $\phi: \mathbb{B}_\infty \to \mathbb{B}_\infty$ need not be open.
  Secondly, there are injective holomorphic
  maps $\phi: \mathbb{B}_\infty \to \mathbb{B}_\infty$ satisfying $\lim_{\|z\| \to 1 } \|\phi(z)\| = 1$
  that are not the composition of an automorphism and a linear isometry.
  This is closely related to the existence of non-linear proper holomorphic injections $\phi: \mathbb{B}_d \to \mathbb{B}_{d'}$
  satisfying $\phi(0) = 0$, where $d' > d$; see for instance \cite{DAngelo03}.
  To give a concrete example, let
  \begin{equation*}
    \psi: \ell^2 \to \ell^2 \oplus (\ell^2 \otimes \ell^2), \quad z \mapsto \frac{1}{\sqrt{2}} (z ,z \otimes z),
  \end{equation*}
  let $U: \ell^2 \oplus (\ell^2 \otimes \ell^2) \to \ell^2$ be a unitary, and set $\phi = U \psi$.
  Then $\phi$ is a holomorphic injection satisfying $\|\phi(z)\|^2 = \frac{1}{2} (\|z\|^2 + \|z\|^4)$,
  so $\phi$ takes $\mathbb{B}_\infty$ to $\mathbb{B}_\infty$ and $\lim_{\|z\| \to 1} \|\phi(z)\| = 1$.
  But $\phi(0) = 0$ and $\phi$ is not linear, so $\phi$ is not of the form $\phi = \varphi_a V$
  for some linear isometry $V$.

  However, note that since $\phi(0) = 0$, the map $\phi$ cannot implement a co-isometric weighted composition operator on a unitarily invariant space by the remarks above.
\end{remark}

\bibliographystyle{amsplain}
\bibliography{literature}

\end{document}